\documentclass[12pt]{amsart}
\usepackage{amsmath}
\usepackage{amssymb}
\usepackage{amsfonts}
\usepackage{textcomp}
\usepackage{amsthm}
\usepackage{mathrsfs}
\usepackage[latin1]{inputenc}
\usepackage[all]{xy}
\usepackage{hyperref}
\usepackage{url}
\usepackage{graphicx}
\usepackage{tikz-cd} 
\usepackage[paperheight=11in,paperwidth=8.5in,left=1in,top=1.25in,right=1in,bottom=1.25in,]{geometry}

\usepackage{empheq}           

\newcommand{\CC}{\overline{\mathbb{C}}}
\newcommand{\Om}[1]{\Omega^{#1}}
\newcommand{\Omc}[1]{\Omega^{#1}_c}
\newcommand{\Z}[1]{Z^{#1}}
\newcommand{\OZ}[1]{\overline{\!Z}^{\,#1}}
\newcommand{\Zc}[1]{Z^{#1}_c}
\newcommand{\B}[1]{B^{#1}}
\newcommand{\Bc}[1]{B^{#1}_c}
\newcommand{\Hc}[1]{H^{#1}_c}
\newcommand{\HR}[1]{H^{#1}}
\newcommand{\OH}[1]{\overline{\!H}^{\,#1}}

\newcommand{\ov}[1]{\overline{#1}}
\newcommand{\Ker}{{\rm Ker}}
\newcommand{\Ima}{{\rm Im}}

\newtheorem{theorem}{Theorem}[section]
\newtheorem{lem}{Lemma}[section]
\newtheorem{prop}{Proposition}[section]
\newtheorem{defi}{Definition}[section]
\newtheorem{cor}{Corollary}[section]

\begin{document}
\begin{abstract}
We will show that any open Riemann surface $M$ of finite genus is biholomorphic to an open set of a compact Riemann surface. Moreover, we will introduce a quotient space of forms in $M$ that determines if $M$ has finite genus and also the minimal genus where $M$ can be holomorphically embedded.
\end{abstract}
\title{A differential form approach to the genus of Open Riemann surfaces}
\author{Franco Vargas Pallete,\, Jes\'{u}s Zapata Samanez}
\email{franco@math.ias.edu}
\email{jesus.zapata@pucp.pe}
\address{School of Mathematics  \\
 Institute for Advanced Study \\
114 MOS, 1 Einstein Drive \\
Princeton, NJ 08540\\
U.S.A.}
\address{Secci\'{o}n de Matem\'{a}ticas\\ Departamento de Ciencias\\ Pontificia Universidad Cat\'olica del Per\'u\\ Lima, Per\'u}
\thanks{Vargas Pallete was partially funded by the Proyecto de Apoyo a la Iniciaci\'{o}n Cient\'{i}fica (PAIN) from the Pontificia Universidad Cat\'{o}lica del Per\'{u} (PUCP) and by the Scholarship for the IMCA Project of the Brazil-Per\'{u} Covenant, and is also supported by the Minerva Research Foundation. Zapata Samanez is supported by the Departamento Acad\'emico de Ciencias of the Pontificia Universidad Cat\'olica del Per\'u (PUCP)}
\maketitle

\section{Introduction}

Even when de Rham cohomology of an open surface is infinite in dimension $1$ (and hence a bit complicated), there is a sufficient and necessary condition in the language of differential $1$-forms for the problem of embedding the surface into the Riemann sphere (Koebe's Generalized Uniformization Theorem). In this article we generalize this sufficient and necessary condition for surfaces with non-zero genus in terms of the dimension of certain quotient space of $1$-forms. This is also in general a necessary condition for the problem of embedding a open $n$-manifold into a compact $n$-manifold.

The article is organized as follows. In Section \ref{sec:background} we set some classical notation and facts about differential forms, and we also define a quotient space using either closed differential forms that are either exact outside a compact set or with compact support. We observe that these spaces are canonically isomorphic. In Section \ref{sec:ck} we study $c_k$, the dimension of our quotient space of $k$-forms. We show that $c_k$ is a topological invariant, is non-decreasing under inclusion and additive under connected sum (except for $k=0,n$ where $n$ is the dimension of the total space). In Section \ref{sec:genus}  we study the relationship between $c_1(X)$ and the genus of a surface $X$, namely that $c_1(X)$ is twice the genus. We also prove that for any given Riemann surface structure on $X$ there is a holomorphic embedding into some compact Riemann surface of the same genus. We finalize with some comments and applications.

\textbf{Acknowledgements} The authors would like to thank the Institute for Advanced Study for hosting them while completing this article. 
\section{Background}\label{sec:background}

Given a differentiable manifold $M$ let us define the many space of $k$-forms that will appear through the article. Denote by $\Om{k}(M)$ the space of real-valued skew-symmetric differential $k$-forms of $M$, where we define the exterior derivative operator $d: \Om{k}(M) \rightarrow \Om{k+1}(M)$ that satisfies $d\circ d =0$. Hence we can define the spaces of closed and exact $k$-forms $\Z{k}(M) = \Ker(\Om{k}(M) \overset{d}{\rightarrow} \Om{k+1}(M)) ,\,\B{k}(M) = \Ima(\Om{k-1}(M) \overset{d}{\rightarrow} \Om{k}(M))$ where $\B{k}(M) \subseteq \Z{k}(M)$.

\begin{defi}
Denote by $\OZ{k}(M)$ the space of $k$-closed form of $M$ such that are exact outside a compact set. That is
\[\OZ{k}(M) = \lbrace \lambda\in \Z{k}(M) \,|\, \exists K \subseteq M \text{ compact},\, \nu\in \Om{k-1}(M\setminus K),\, d\nu=\lambda|_{M\setminus K} \rbrace\]
\end{defi}

Note that we have the inclusion $\B{k}(M) \subseteq \OZ{k}(M)$, so we can define as well

\begin{defi}\label{def:H1}
$\OH{k}(M) = \OZ{k}(M)/\B{k}(M)$, $c_k(M) = {\rm dim}_{\mathbb{R}}\,\OH{k}(M)$
\end{defi}

We can obtain a quotient real vector space canonically isomorphic to $\OH{k}(M)$ by taking differentials with compact support. Let us denote then by $\Omc{k}(M) \subset \Om{k}(M)$ the space of differential $k$-forms of $M$ with compact support. We have the following commutative diagram

\begin{equation}\label{diag:dcommutes}
\begin{tikzcd}
\Omc{k}(M) \arrow[r, hook, "i"] \arrow[d, "d"]
& \Om{k}(M) \arrow[d,"d"] \\
\Omc{k+1}(M) \arrow[r, hook, "i"]
& \Om{k+1}(M)
\end{tikzcd}
\end{equation}
where $d$ is again the exterior derivative and the horizontal maps are the inclusions $\Omc{k}(M) \subset \Om{k}(M) $. In case that $M$ is compact these inclusions are identities. We can then define the closed forms with compact support $\Zc{k}(M)= \Ker(\Omc{k}(M) \overset{d}{\rightarrow} \Omc{k+1}(M))$ and the exact forms with compact support $ \Bc{k}(M) = \Ima(\Omc{k-1}(M) \overset{d}{\rightarrow} \Omc{k}(M))$. If we were going to take the quotient $\Zc{k}(M)/\Bc{k}(M)$ we will rescue the de Rham cohomology with compact support. Instead, we will look at the space of forms with compact support that are exact.

\begin{defi}\label{def:Bck}
Denote by $\ov{\Bc{k}}(M)$ the space of forms with compact support that are the exterior derivative of a form that does not need to have compact support. That is, $\ov{\Bc{k}}(M) = i^{-1}(\B{k}(M)) = \B{k}(M) \cap \Omc{k}(M)$
\end{defi}

Given the commutativity of Diagram \ref{diag:dcommutes} and that $d\circ d \equiv 0$, we have the following inclusions

\begin{equation}
\begin{tikzcd}
\Bc{k}(M) \arrow[r,hook]\arrow[d,hook]
& \B{k}(M) \arrow[dr,hook] \\
\ov{\Bc{k}}(M)  \arrow[r,hook]
&\Zc{k}(M) \arrow [r,hook]
&\Z{k}(M) \\
\end{tikzcd}
\end{equation}
which in particular allows us to define the quotient $\Zc{k}(M)/\ov{\Bc{k}}(M)$. Let us see that this quotient is canonically isomorphic to the quotient in Definition \ref{def:H1} via the inclusion map $\Zc{k}(M) \hookrightarrow \OZ{k}(M)$.

\begin{lem}\label{lem:incisom}
The inclusion map $\Zc{k}(M) \hookrightarrow \OZ{k}(M)$ induces an isomorphism $\Zc{k}(M)/\ov{\Bc{k}}(M) \overset{\cong}\rightarrow \OZ{k}(M)/\B{k}(M)$.
\end{lem}

\begin{proof}
Notice first that the preimage of $\B{k}(M)$ under the inclusion map is $\ov{\Bc{k}}(M)$, so the quotient map $\Zc{k}(M)/\ov{\Bc{k}}(M) \rightarrow \OZ{k}(M)/\B{k}(M)$ is well defined and injective. To show surjectivity, take any form $\omega \in  \OZ{k}(M)$. Since $\omega$ is exact outside a compact set, by using a bump function we can write $\omega = d\alpha + \omega_0$, where $\alpha\in\Omega^{k-1}(M)$ and $\omega_0\in Z_c^k(M)$. But this exactly says that the image of $[\omega_0]$ is $[\omega]$.
\end{proof}

Under this natural isomorphism we will represent the vector space $\OH{k}(M)$ as a quotient of the space of forms more appropriate to our goals. Forms with compact support are convenient for their relation with the restriction and extension maps of forms, that we describe now. Differentiable forms have the natural map \textit{restriction}(r) for $U \subset M$ a open set that makes the following diagram commute

\begin{equation}\label{diag:restriction}
\begin{tikzcd}
\Om{k}(M) \arrow[r, "r"] \arrow[d, "d"]
& \Om{k}(U) \arrow[d,"d"] \\
\Om{k+1}(M) \arrow[r, "r"]
& \Om{k+1}(U)
\end{tikzcd}
\end{equation}

Likewise, compactly supported forms have the natural map \textit{extension}(e) which is injective at all ranks and makes the following diagram commute

\begin{equation}\label{diag:extension}
\begin{tikzcd}
\Omc{k}(U) \arrow[r, hook, "e"] \arrow[d, "d"]
& \Omc{k}(M) \arrow[d,"d"] \\
\Omc{k+1}(U) \arrow[r, hook, "e"]
& \Omc{k+1}(M)
\end{tikzcd}
\end{equation}

Given Diagrams \ref{diag:restriction} and \ref{diag:extension} we know that restriction and extension preserve closed and exact forms.

We can describe the interaction between restriction and extension in the following commutative diagram

\begin{equation}\label{diag:cancellation}
\begin{tikzcd}
\Omc{k}(U) \arrow[r, hook, "e"] \arrow[d, hook, "i"]
& \Omc{k}(M) \arrow[d, hook, "i"] \\
\Om{k}(U) \arrow[r, leftarrow, "r"]
& \Om{k}(M)
\end{tikzcd}
\end{equation}

We can quickly compare the dimension $c_k$ with the Betti numbers $b_k = {\rm dim}_{\mathbb{R}} \frac{\Z{k}(M)}{\B{k}(M)}$.

\begin{lem}\label{lem:ckvsbk}For any differentiable manifold $M$ and non-negative integer $k$, $c_k(M) \leq b_k(M)$. Moreover, if $M$ is compact, then $c_k(M)=b_k(M)$.
\end{lem}
\begin{proof}
The inequality follows from the injective map $\frac{\OZ{k}(M)}{\B{k}(M)} \hookrightarrow \frac{\Z{k}(M)}{\B{k}(M)}$ induced by the inclusion $\B{k}(M) \subseteq \OZ{k}(M) \subseteq \Z{k}(M)$, where the last inclusion is an equality if $M$ is compact.
\end{proof}

Forms behave well with respect to pullbacks of smooth functions. Since at times we will be dealing with forms with compact support, we will restrict ourselves to proper function between manifold. Let them $f:M\rightarrow N$ be a smooth proper function, the \textit{pullback} of $f$ is the natural map of chains $f^*$ that makes the following diagrams commute.

\begin{eqnarray}
\begin{tikzcd}\label{diag:pullbackf}
\Om{k}(N) \arrow[r, "f^*"] \arrow[d, "d"]
& \Om{k}(M) \arrow[d,"d"] \\
\Om{k+1}(N) \arrow[r, "f^*"]
& \Om{k+1}(M)
\end{tikzcd}\quad
\begin{tikzcd}
\Omc{k}(N) \arrow[r, "f^*"] \arrow[d, "d"]
& \Omc{k}(M) \arrow[d,"d"] \\
\Omc{k+1}(N) \arrow[r, "f^*"]
& \Omc{k+1}(M)
\end{tikzcd}
\end{eqnarray}
which in particular concludes that $f^*$ can be defined in cohomology as map from $\frac{\Z{k}(N)}{\B{k}(N)}$ to $\frac{\Z{k}(M)}{\B{k}(M)}$.

Moreover, if $f,g:M\rightarrow N$ are homotopic, we know that there is a algebraic homotopy between $f^*$ and $g^*$ given by a linear map $P:\Om{k+1}(N) \rightarrow \Om{k}(M)$ that commutes with the inclusion of (\ref{diag:dcommutes}) and satisfies

\begin{equation}\label{eq:homalg}
f^* (\omega) - g^*(\omega) = P(d\omega) + dP(\omega)
\end{equation}
for all $\omega\in\Om{k}(M)$. In particular the maps $f^*, g^*:  \frac{\Z{k}(N)}{\B{k}(N)} \rightarrow \frac{\Z{k}(M)}{\B{k}(M)}$ are equal. The analogue statement is true if we switch to de Rham cohomology with compact support.

For the topological and geometric significance of the $c_k$ invariants (particularly $c_1$) we will focus on Riemann surfaces. We will adopt the notation of $\Sigma_{g,n}$ for the surface with genus $g$ and $n$ boundary components.

A connected Riemann surface $X$ is said to be \textit{planar} or \textit{schlichtartig} if every closed $1$-form on $X$ with compact support is exact. Under our notation this is the same as $c_1(X)$ vanishing, for which we have the Generalized Uniformization Theorem of Koebe (see for instance \cite{Simha89}).

\begin{theorem}[Generalized Uniformization Theorem]\label{thm:Simha}
Every planar Riemann surface is biholomorphic (i.e. conformally equivalent) to an
open subset of the Riemann sphere $\CC$.
\end{theorem}

\section{Properties of $c_k$}\label{sec:ck}

Our first proposition is to back up the claim that $c_k$ is an invariant.

\begin{prop}\label{prop:cinv}
Let $M,\,N$ be differentiable manifolds such that they are homeomorphic. Then $c_k(M) = c_k(N)$
\end{prop}
\begin{proof}
Denote by $h: M \rightarrow N$ a homeomorphism between $M$ and $N$, $f:M\rightarrow N$ a differentiable proper function homotopic to $h$ and $g: N \rightarrow M$ a differentiable proper function homotopic to $h^{-1}$. Hence $g\circ f \sim_{\rm hom} id_M$ are differentiable and homotopic, they define the same map at the level of cohomology. Moreover, since $f^*$ is proper and commutes with the exterior derivative $d$, we have that $f^* (\OZ{k}(N))=\OZ{k}(M)$ (with the analogue statement for $g$). Since the following diagram commutes

\begin{equation}\label{diag:inversefg}
\begin{tikzcd}
\frac{\OZ{k}(M)}{\B{k}(M)} \arrow[r, rightarrow, "g^*"] \arrow[d, hook, "i"]
& \frac{\OZ{k}(N)}{\B{k}(N)} \arrow[r, rightarrow, "f^*"]\arrow[d, hook, "i"] 
&\frac{\OZ{k}(M)}{\B{k}(M)} \arrow[d, hook, "i"]\\
\frac{\Z{k}(M)}{\B{k}(M)} \arrow[r, rightarrow, "g^*"]
& \frac{\Z{k}(N)}{\B{k}(N)} \arrow[r, rightarrow, "f^*"]
& \frac{\Z{k}(M)}{\B{k}(M)}
\end{tikzcd}
\end{equation}
then the top composition $f^*\circ g^*$ must be the identity because the bottom composition $f^*\circ g^* = id^*_M$ is the identity and the down arrows are injective (as we saw in Lemma \ref{lem:ckvsbk}).

\end{proof}

The next proposition shows that $c_k$ is non-decreasing under inclusion. This together with Lemma \ref{lem:ckvsbk} implies that a necessary condition for a open $n$-manifold to embed into a compact manifold is that all $c_k$ must be finite. This also restrict the possibilities for the compact manifold since $c_k$ are lower bounds for the Betti numbers. Conversely, this gives a broad family of manifolds with finite $c_k$, namely open sets of compact manifolds.

\begin{prop}\label{prop:cnondec}
Let $U \subset M$ be a open set of a differentiable manifold $M$. Then $c_k(U) \leq c_k(M),\, \forall \, k\geq 0$.
\end{prop}

\begin{proof}
The result will be deduced from the following commutative diagram

\begin{equation}\label{diag:nondec}
\begin{tikzcd}
\Zc{k}(U) \arrow[r,hook, "e"]\arrow[d,two heads]
& \Zc{k}(M) \arrow[d, two heads] \\
\frac{\Zc{k}(U)}{e^{-1}(\ov{\Bc{k}}(M))}  \arrow[r,hook, "e"] \arrow[d, two heads]
&\frac{\Zc{k}(M)}{\ov{\Bc{k}}(M)} \\
\frac{\Zc{k}(U)}{\ov{\Bc{k}}(U)}
\end{tikzcd}
\end{equation}
where we still have to justify the commutativity, as well as the injectivities and surjectivities claimed.

The first map is the extension map for closed forms, which gives us an injective map from $\Zc{k}(U)$ into $\Zc{k}(M)$. The down maps are the respective quotient maps which are obviously surjective, while the quotient map of $e$ (also called $e$) is injective thanks to the injectivity of $e$. As for the final arrow, by using Diagram \ref{diag:cancellation} applied to closed forms and the definition of $\ov{\Bc{k}}(M)$ (\ref{def:Bck}) we can notice that

\begin{equation}
e^{-1} (\ov{\Bc{k}}(M)) = i^{-1}(r\circ i(\ov{\Bc{k}(M)})) \subset i^{-1}(r(\B{k}(M))) \subset i^{-1}(\B{k}(U)) = \ov{\Bc{k}(U)},
\end{equation}
which tell us that the last vertical arrow is well defined and a surjection. Finally, the propostion follows from

\begin{equation}
c_k(U) = \dim\displaystyle\frac{\Zc{k}(U)}{\ov{\Bc{k}}(U)} \leq \dim\displaystyle\frac{\Zc{k}(U)}{e^{-1}(\ov{\Bc{k}}(M))} \leq \dim\displaystyle\frac{\Zc{k}(M)}{\ov{\Bc{k}}(M)} = c_k(M)
\end{equation}
since $\displaystyle\frac{\Zc{k}(U)}{e^{-1}(\ov{\Bc{k}}(M))}$ has a surjective map to $\displaystyle\frac{\Zc{k}(U)}{\ov{\Bc{k}}(U)}$ and an injective map to $\displaystyle\frac{\Zc{k}(M)}{\ov{\Bc{k}}(M)}$.
\end{proof}

Observe then that from Diagram \ref{diag:nondec} that if $c_k(U)$ is finite dimensional then there exists $R^U_k$, a $c_k$-dimensional subspace of $\Zc{k}(U)$, such that $R_k^U \overset{e}{\rightarrow} \Zc{k}(M)$ passes through the quotients as a non-cannonical injection $\displaystyle\frac{\Zc{k}(U)}{\ov{\Bc{k}}(U)} \overset{e_U}{\hookrightarrow} \displaystyle\frac{\Zc{k}(M)}{\ov{\Bc{k}}(M)} $. An easy consequence of this lemma and Lemma \ref{lem:ckvsbk} is the following corollary.

\begin{cor}\label{cor:flat}
Let $U\subset S^2$ be an open set. Then $c_1(U)=0$.
\end{cor}

The main result of this paper can be thought as a converse of this statement for $2$-dimensional manifold. Before restricting ourselves to the $2$-dimensional case, let us address two more results for $c_k$, starting with the following lemma.

\begin{lem}\label{lem:minusB}
Let $M$ be a differentianle $n$-manifold, $0<k \neq n$ and integer and $B$ a $n$-ball in $M$ with compact closure. Then $c_k(M\setminus \ov{B}) = c_k(M)$.
\end{lem}
\begin{proof}
Given Proposition \ref{prop:cnondec} with $U = M\setminus \ov{B}$ and the comment right after it, the result will follow after showing that $R^U_k \overset{e_U}\hookrightarrow \displaystyle\frac{\Zc{k}(M)}{\ov{\Bc{k}}(M)}$ is a surjection (since it is already injective). Choose $B_4 \overset{\circ}\supset B_3 \overset{\circ}\supset B_0 = B$ balls with compact closure in $M$, $\varphi\in \Omc{0}(M)$ a function with support in $B_4$ and equal to $1$ in $B_3$. Furthermore, for a given element in $\displaystyle\frac{\Zc{k}(M)}{\ov{\Bc{k}}(M)}$ represented by $\omega \in \Zc{k}(M)$, take $\eta\in \Om{k-1}(B_4)$ such that $\omega = d\eta$ in $B_4$ (this is possible since balls are contractible). Then $\omega-d(\varphi\eta)$ vanishes in $B_3$, so it is the extension of a form in $\Zc{k}(U = M\setminus \ov{B_0})$, which we will keep denoting by the same expression. Then exists $\mu \in \Om{k-1}(U)$ such that $\ov{\omega} = \omega - d(\varphi\eta) - d(\mu) \in R^U_k$. We would like to say that the image of $\ov{\omega}$ under $e_U$ is $\omega$, but this is true if only if we can select $\mu$ such that $d(\mu) \in \ov{\Bc{k}}(M)$, which in turn is true if we can pick $\mu$ that is the restriction of a $(k-1)$-form from $M$.

In order to do so, since the support of $d(\mu)$ is a compact in $U = M\setminus \ov{B_0}$, we can pick $B_3\supset B_2 \supset B_1 \supset B_0$ such that $d(\mu)$ vanishes in $B_2\setminus \ov{B_0}$, so then $\mu\in \Z{k-1}(B_2\setminus \ov{B_0})$. For $k>1$, since for usual real cohomology ${\rm H}^\ell(M) = \frac{\Z{\ell}(M)}{\B{\ell}(M)}$ we know that for $k\neq n$, ${\rm H}^{k-1} (B_2\setminus \ov{B_0}) = 0$, then there is $\nu\in \Om{k-2}(B_2\setminus \ov{B_0})$ such that $d\nu = \mu$. Take then $\phi \in \Omc{k-1}(B_2)$ equal to $1$ in $B_1$ so then $\mu-d(\phi\nu)$ can be extended as $\mu$ outside of $B_2$ and as $0$ inside of $B_0$, so $\ov{\omega} = \omega - d(\varphi\eta) - d(\mu - d(\phi\nu))$ has image $\omega$ under $e_U$. If $k=1$, $\mu$ is a constant function in $B_2\setminus \ov{B_0}$, since $d\mu = 0$ and $B_2\setminus \ov{B_0}$ is connected. Then $\mu$ extends as a constant inside $B_0$, and we have again that $\ov{\omega} = \omega - d(\varphi\eta) - d(\mu)$ has image $\omega$ under $e_U$.

\end{proof}

We have the following easy application. For a surface $\Sigma$ with genus $g$ and $n$ punctures, $c_1 = 2g$. Lemma \ref{lem:minusB} tells us that $\Sigma$ has the same $c_1$ invariant as the closed surface of genus $g$, and for this surface $c_1$ coincides with the first Betti number $b_1$, which is equal to $2g$.

\begin{prop}\label{prop:additive} Let $M = N_1 \# N_2$ be the connected sum of two $n$-differentiable manifolds $N_1, N_2$. Then $c_k(M) = c_k(N_1) + c_k (N_2)$ for $0< k \neq n$.
\end{prop}
\begin{proof}
First of all, we can assume that $c_k(N_1), c_k(N_2)$ are both finite, since otherwise the result follows from Proposition \ref{prop:cnondec}. Now, because of the same comment at the end of the proof of this Proposition \ref{prop:cnondec}, we have the following map

\begin{equation}\label{eq:summap}
\displaystyle\frac{\Zc{k}(U)}{\ov{\Bc{k}}(U)} \oplus \displaystyle\frac{\Zc{k}(V)}{\ov{\Bc{k}}(V)}  \xrightarrow{e_U \oplus e_V} \displaystyle\frac{\Zc{k}(M)}{\ov{\Bc{k}}(M)}
\end{equation}
where $U, V$ are the copies of $N_1\setminus \ov{B}, N_2\setminus \ov{B}$ in $M = N_1 \# N_2$, for balls in $N_1, N_2$. Moreover, we can assume that the support of the elements in $R^U_k, R^V_k$ do not intersect the $S^{n-1}\times]-1,1[$ region used to glue $M$. Thanks to Lemma \ref{lem:minusB} the result will follow if we show that $e_U \oplus e_V$ is a bijection.

Let us show first that $e_U \oplus e_V$ is injective. Assume by contradiction that there are $\omega_U, \omega_V$ elements in $R^U_k, R^V_k$ (respectively) such that the image of $\omega_U + \omega_V$ is $0=\displaystyle\frac{\Zc{k}(M)}{\ov{\Bc{k}}(M)}$ in \ref{eq:summap}. Then there exists $\eta\in\Om{k-1}(M)$ such that $(e_U \oplus e_V)(\omega_U + \omega_V) = d\eta$ in $\Omc{k}(M)$. But since $\omega_U, \omega_V$ are supported away from the gluing region $S^{n-1}\times]-1,1[$, the restriction maps $r_U, r_V$ and Diagram \ref{diag:cancellation} gives us the equations

\begin{eqnarray}
\omega_U = d(r_U(\eta))\\
\omega_V = d(r_V(\eta))
\end{eqnarray}
which implies that $\omega_U, \omega_V$ belong to $\ov{\Bc{k}}(U), \ov{\Bc{k}}(V)$, respectively. This concludes the proof that $e_u\oplus e_V$ is injective.

Now let us prove that $e_U \oplus e_V$ in \ref{eq:summap} is surjective. Similar to the proof of Lemma \ref{lem:minusB}, we want to show that $\omega \in \Zc{k}(M)$ can be written as $\omega_1 + \omega_2 + d\eta$, with $\omega_i$ with compact support in $N_i\setminus\ov{B},\, i=1,2$, and $\eta\in\Z{k-1}(M)$. Let us divide in two cases:

\begin{itemize}
\item For $k\neq n-2$, recall that for compactly supported de Rham cohomology $\Hc{k}(M) = \frac{\Zc{k}(M)}{\Bc{k}(M)}$ we have the following Mayer-Vietoris exact sequence induced by the extension maps:

\begin{equation}
\ldots\rightarrow  \Hc{k}(S^{n-1}\times]-1,1[)  \rightarrow \Hc{k}(N_1\setminus \ov{B}) \oplus \Hc{k}(N_2\setminus \ov{B}) \rightarrow \Hc{k}(M) \rightarrow \Hc{k+1}(S^{n-1}\times]-1,1[) \rightarrow \ldots
\end{equation}
Then since $k\neq n-2$, $\Hc{k+1}(S^{n-1}\times]-1,1[)=0$ so $[\omega]\in\Hc{k}(M)$ is the image of $([\omega_1], [\omega_2])\in \Hc{k}(N_1\setminus \ov{B}) \oplus \Hc{k}(N_2\setminus \ov{B})$. This means that $\omega$ is equal to $\omega_1 + \omega_2 + d\eta$, where we are using the same notation of $\omega_i$ for their extensions to $\Omc{k}(M)$, and $\eta\in\Z{k-1}(M)$.

\item For $k=n-2$ recall that $\HR{n-2}(S^{n-1}\times]-1,1[)= 0$. Then for a given $\omega\in\Zc{k}(M)$ there exists $\rho\in\Z{k-1}(S^{n-1}\times]-1,1[)$ such that $d\rho=\omega$ in $S^{n-1}\times]-1,1[$. Then as in Lemma \ref{lem:minusB} we can choose $\varphi$ function with compact support in $(S^{n-1}\times]-1,1[)$ such that $\omega-d(\varphi\rho)$ vanishes in $S^{n-1}\times]-\frac12,\frac12[$. Then take $\eta=\varphi\rho$

Now again as in Lemma \ref{lem:minusB} we find $\mu_i\in \Om{k-1}(N_i\setminus \ov{B})$ such that $\ov{\omega_i} = \omega_i-d\mu_i \in R^{N_i\setminus\ov{B}}_k$. The proof will be complete as soon as we manage to pick $\mu_i$ that extends to all $M$. But as in Lemma \ref{lem:minusB} we extend either as $0$ after taking out a term $d(\phi\nu), \phi\nu\in\Om{k-2}(M)$ or as a constant function. Then we will have that $e_U\oplus e_V$ is surjective.
\end{itemize}

\end{proof}

\section{Genus of an open surface}\label{sec:genus}

Observe we can use Proposition \ref{prop:additive} and Corollary \ref{cor:flat} to conclude that if $\Sigma$ is made out of the connected sum of a genus $g$ surface and some flat surfaces, then $c_1(\Sigma) = 2g$. Then we can ask ourselves if that is the case whenever $c_1$ is finite, which is the main result of this article.

\begin{theorem}\label{thm:genus}
Let $X$ be a Riemann surface such that $c_1(X)$ is finite. Then $c_1(X)$ es even and $X$ can be conformally embedded into a compact Riemann surface of genus $g=c_1(X)/2$. Moreover, $g$ this is the smallest genus where $X$ can be embedded and in fact $g$ is the genus of $X$.
\end{theorem}
\begin{proof}
Take $M_0 \overset{\circ}\subset M_1 \overset{\circ}\subset M_2 \overset{\circ}\subset \ldots$ an exhaustion by compact submanifolds of $X$. Using Proposition \ref{prop:cnondec} we know that $c_1(M_i)$ is a non-decresing sequence of integers bounded by $c_1(X)$, so it is eventually constant equal to $C \leq c_1(X)$. Moreover, $c_1(M_i)/2$ is the genus of $M_i$, so $C/2$ is a way to obtain the genus of $X$ since its value does not depend on the exhaustion (recall that genus is monotone for compact surfaces with boundary).

\textbf{Claim:} $c_1(X) = C$.
\begin{proof}
Fix a base $\lbrace [\omega_1],\ldots,[\omega_{c_1(X)}] \rbrace$ of $\displaystyle\frac{\Zc{1}(X)}{\ov{\Bc{1}}(X)}$ with representatives $\lbrace \omega_1,\ldots,\omega_{c_1(X)} \rbrace \subset \Zc{1}(X)$. Then there is $n$ big enough such that $M_n$ contains the support of all the elements $\lbrace \omega_1,\ldots,\omega_{c_1(X)} \rbrace$. Then, as in Lemma \ref{lem:minusB} and Proposition \ref{prop:additive} and using that every component of $\partial M_n$ is $S^1$, we can show that $\displaystyle\frac{\Zc{1}(M_n)}{\ov{\Bc{1}}(M_n)} \hookrightarrow \displaystyle\frac{\Zc{1}(X)}{\ov{\Bc{1}}(X)}$ is surjective, which in turn tells us that $c_1(X) \leq C$, and this completes the proof of the claim since we already knew $C \leq c_1(X)$.
\end{proof}

Take now $M = M_n$ such that $c_1(M) = c_1(X)$.

\textbf{Claim:} For every component $N$ of $X\setminus M$, $c_1(N)=0$.
\begin{proof}
Recall that for a compact surface $\Sigma$ (with boundary) and a collection of $\ell$ components $\mathcal{C}$ of $\partial\Sigma$ we can write $\Sigma = \underline{\Sigma} \# \Sigma_{0,\ell}$ where the components of $\Sigma_{0,\ell}$ correspond to $C$. This can be done by cutting $\Sigma$ along a separating curve that has the elements of $\mathcal{C}$ at one side that is flat as well. Taking that $M$ and $M_{n+1} \cap N$ share $\ell>0$ boundary components, we can use the previous fact to write $M$ and $M_{n+1} \cap N$ as $\underline{M} \# \Sigma_{0,\ell}$, $\underline{M_{n+1}\cap N} \# \Sigma_{0,\ell}$. In particular $N = \underline{N} \# \Sigma_{0,\ell}$, where $\underline{N}$ is $\underline{M_{n+1} \cap N}$ glued with $(N\setminus M_{n+1})$. We can also then write $M\cup N $ as

\begin{equation}
M\cup N = \underline{M} \# \Sigma_{\ell-1, 0} \# \underline{N}
\end{equation}
where $\Sigma_{\ell-1,0}$ is the result of gluing the two copies of $\Sigma_{0,\ell}$ by their boundaries.

Since from Corollary \ref{cor:flat} $c_1(\Sigma_{0,\ell}) = 0$, Proposition \ref{prop:additive} tells us that $c_1(M) = c_1(\underline{M}), c_1(N) = c_1(\underline{N})$. Furthermone, Proposition \ref{prop:additive} also tells us that $c_1(M\cup N) = c_1(\underline{M}) + 2(\ell-1) + c_1(\underline{N})$. But since $c_1(M) = c_1(\underline{M})$ is equal to $c_1(\Sigma)$, Proposition \ref{prop:cnondec} gives us $c_1(M\cup N) = c_1(M)$. Then it follows that $c_1(N)=0$ and $\ell =1$.
\end{proof}

Now, for each of the finitely many components $N$ of $X\setminus M$ we can apply Theorem \ref{thm:Simha} to obtain a conformal embedding $N \xhookrightarrow{i_N} \ov{\mathbb{C}}$. Moreover, for each boundary component $\gamma$ of $M$ we can take an annulus neighbourhood $V_\gamma$ and glue $M\cup V_\gamma$ along $V_\gamma$ by $i_N$ to the component of $\ov{\mathbb{C}}\setminus i_N(\gamma)$ that contains $i_n(V_\gamma)$. If we do this for every component $N$ of $X\setminus M$ we will obtaind a Riemann surface $\Sigma$ because the gluing maps $i_N$ were conformal. Moreover, we have a natural conformal embedding $X\hookrightarrow \Sigma$ that is equal to the identity in $M\cup V_\gamma$ and equal to $i_N$ on each $N$. And because of the Jordan's curve theorem, the components of $\ov{\mathbb{C}}\setminus i_N(\gamma)$ are disks, so $\Sigma$ is compact. Finally, since $M$ is obtained from $\Sigma$ after removing some disks, then Lemma \ref{lem:minusB} tells us that $c_1(\Sigma) = c_1(M) = c_1(X)$, so then $c_1(X)=2g$ for $g$ the genus of $\Sigma$. In light of Proposition \ref{prop:cnondec} this is the smallest genus where $X$ can be embedded.

\end{proof}

Note that for the holomorphic embedding $X\hookrightarrow \Sigma$ the Riemann surface $\Sigma$ is not determined uniquely for a fixed open Riemann surface $X$. Indeed, if $\Sigma\setminus X$ has interior, we can fix any non-trivial holomorphic quadratic differential $\phi$ of $\Sigma$ and a Beltrami differential $\mu$ supported in $\Sigma\setminus X$ such that $\langle\phi, \mu \rangle \neq 0$. Then if $\Sigma_t$ is the Riemann surface obtained by solving the Beltrami equation associated to $t\mu$, then the tangent vector of ${\Sigma_t}$ at $\Sigma$ is given by $\mu$, which is a non-zero tangent vector of Teichm\"{u}ller space since it has non-zero pairing with $\phi$ who is an element of the cotangent space. Then $\Sigma_t$ is a non-constant path in Teichm\"{u}ller space, and since $t\mu$ vanishes in $X$, each of them admits inside a holomorphic copy of $X$.

On the other hand, if $X$ is of finite type then holomorphic embeddings are quite restrictive, as the following corollary concludes.

\begin{cor}
Let $X$ be biholomorphic to a compact Riemann surface $\Sigma$ with $n$ punctures and assume that there is a holomorphic embedding of $X$ to a (maybe open) Riemann surface $N$ with finite genus. Then $N$ is biholomorphic to $\Sigma$ with $n'$ punctures, where $n'\leq n$.
\end{cor}
\begin{proof}
Since $N$ has finite genus, there is a holomorphic embedding $N\hookrightarrow\Sigma'$ where $\Sigma'$ is a compact Riemann surface. Then $X$ embeds into $\Sigma'$ by composing the maps. By a classical result, this embedding extends to a biholomorphic map $\Sigma \rightarrow \Sigma'$. And since $N$ contains the image of $X$, we see that $N$ is also of finite type with no more than $n$ punctures.
\end{proof}

\begin{cor}
Let $X$ be a Riemann surface homeomorphic to an open set of a compact surface. Then $X$ can be holomorphically embedded into a compact Riemann surface.
\end{cor}
\begin{proof}
Because of Lemma \ref{lem:ckvsbk} and Proposition \ref{prop:cinv} we have that $c_1(X)$ is finite, so the claim follows from Theorem \ref{thm:genus}
\end{proof}

\begin{cor}\label{cor:resisom}
Let $X$ be an open Riemann surface of finite genus and $X\hookrightarrow\Sigma$ a holomorphic embedding of $X$ into a compact Riemann surface $\Sigma$ of the same genus. Then the restriction $r:Z^1(\Sigma)\to\,\overline{\!Z}^{\,1}\!(X)$ is well defined and induces an isomorphism $r:H^1(\Sigma)\to\,\overline{\!H}^{\,1}\!(X)$.
\end{cor}
\begin{proof}
Because $c_1(X)=c_1(\Sigma)$, the homomorphism $\dfrac{Z^1_c(X)}{e^{-1}(\,\overline{\!B}^{\,1}_c\!(X))}\xhookrightarrow{\ e\ }H^1(M)$ of diagram (\ref{diag:nondec}) is biyective. That means in particular that any $\omega\in Z^1(\Sigma)$ can be expressed as $\omega=d\alpha+e(\omega_0)$, where $\alpha$ is a smooth function on $\Sigma$ and $\omega_0\in Z^1_c(X)$. Hence $\omega=d\alpha+\omega_0$ in $X$, so the restriction of $\omega$ to $X$ is exact at infinity. Thus the restriction $r:Z^1(\Sigma)\to\,\overline{\!Z}^{\,1}\!(X)$ is well defined and induces a homomorphism $r:H^1(\Sigma)\to\,\overline{\!H}^{\,1}\!(X)$. Next, since $H^1(\Sigma)$ and $\,\overline{\!H}^{\,1}\!(X)$ have the same dimension, it is enough to prove that $r$ is surjective in order to be an isomorphism. That is clear from lemma \ref{lem:incisom} because any $\omega\in\,\overline{\!Z}^{\,1}\!(X)$ can be expressed as $\omega=d\alpha+\omega_0$ where $\alpha$ is a smooth function on $X$ and $\omega_0\in Z^1_c(X)$, so $r[e(\omega_0)]=[\omega]$.
\end{proof}

Another interesting application is an analogous of Hodge Theorem for open Riemann surfaces of finite genus:

\begin{cor}
If $X$ is an open Riemann surface of finite genus, then every class in $\,\overline{\!H}^{\,1}\!(X)=\,\overline{\!Z}^{\,1}\!(X)/B^1(X)$ has a harmonic representative (which is exact at infinity). Moreover, given a holomorphic embedding of $X$ into a compact Riemann surface $\Sigma$ of the same genus, each class in $\,\overline{\!H}^{\,1}\!(X)$ has a unique harmonic representative obtained as the restriction of a harmonic 1-form on $\Sigma$.
\end{cor}
\begin{proof}
Fix a holomorphic embedding of $X$ into a compact Riemann surface $\Sigma$ of the same genus and call by $\mathcal H^1(\Sigma)$ the space of harmonic 1-forms on $\Sigma$. Since the restriction $r:\mathcal H^1(\Sigma)\to\,\overline{\!H}^{\,1}\!(X)$ is the composition of the inclusion $i:\mathcal H^1(\Sigma)\to H^1(\Sigma)$ with the restriction $r:H^1(\Sigma)\to\,\overline{\!H}^{\,1}\!(X)$, and both maps are isomorphisms because of Hodge theorem and corollary \ref{cor:resisom} respectively, we conclude $r:\mathcal H^1(\Sigma)\to\,\overline{\!H}^{\,1}\!(X)$ is an isomorphism. Finally, because of the conformal invariance of harmonic 1-forms, we have $r(\mathcal H^1(\Sigma))\subset\mathcal H^1(X)$ and the result immediately follows.
\end{proof}


\bibliographystyle{amsalpha}
\bibliography{bib}

\providecommand{\bysame}{\leavevmode\hbox to3em{\hrulefill}\thinspace}
\providecommand{\MR}{\relax\ifhmode\unskip\space\fi MR }
\providecommand{\MRhref}[2]{%
  \href{http://www.ams.org/mathscinet-getitem?mr=#1}{#2}
}
\providecommand{\href}[2]{#2}
\begin{thebibliography}{Sim89}

\bibitem[Sim89]{Simha89}
R.~R. Simha, \emph{The uniformisation theorem for planar {R}iemann surfaces},
  Arch. Math. (Basel) \textbf{53} (1989), no.~6, 599--603. \MR{1023976}

\end{thebibliography}
\end{document}